\documentclass[preprint,12pt]{elsarticle}
\usepackage{}

\usepackage{amsfonts}
\usepackage{amsmath}
\usepackage{amssymb}
\usepackage{graphics}
\usepackage{graphicx}
\usepackage{mathrsfs}
\usepackage{indentfirst}
\usepackage{stmaryrd}
\usepackage{latexsym}
\usepackage{xypic}
\newtheorem{definition}{Definition}[section]
\newtheorem{lemma}[definition]{Lemma}
\newtheorem{theorem}[definition]{Theorem}
\newtheorem{proposition}[definition]{Proposition}
\newtheorem{example}[definition]{Example}

\newtheorem{remark}[definition]{Remark}
\newtheorem{corollary}[definition]{Corollary}
\newtheorem{open problem}[definition]{Open problem}
\newproof{proof}{\textbf{Proof}}

\journal{}

\begin{document}

\begin{frontmatter}



\title{\textbf{On profiniteness and Hausdorffness of topological residuated lattices}}


\author{Jiang Yang$^{a}$, Pengfei He$^{b,\ast}$, Juntao Wang$^{c}$ }
\cortext[cor1]{Corresponding author. \\
Email addresses: yangjiangdy@126.com, hepengf1986@126.com, wjt@xsyu.edu.cn }

\address[A]{School of Mathematics, Northwest University, Xi'an, 710127, China}
\address[B]{School of Mathematics and Information Science, Shaanxi Normal University, Xi'an, 710119, China}
\address[C]{School of Science, Xi'an Shiyou University, Xi'an, 710065, China}

\begin{abstract}

The aim of this paper is to study the profiniteness of compact topological residuated lattices and the existence of Hausdorff topological residuated lattices. Firstly, we study profinite residuated lattices and obtain sufficient and necessary conditions for profiniteness in compact topological residuated lattices. These conditions include topological and algebraic characterizations. Moreover, it order to study the existence of Hausdorf topological residuated lattices, we investigate finiteness conditions in residuated lattices. Finally, we investigate linear topological residuated lattices and give the class of residuated lattices that can be endowed with a non-trivial Hausdorff topology.

\end{abstract}


\end{frontmatter}

\section{Introduction}
\label{intro}

Recently, algebraic methods have been applied successfully to non-classical logics via universal algebra and algebraic logic. While Boolean algebras are algebraic semantics for two-valued logics, multiple-valued logic algebras serve as algebraic semantics for non-classical logics. Among various multiple-valued logic algebras, residuted lattices are very basic and important algebraic structures. In fact, residuated lattices are structures that have been studied by algebraists since the 1930s in \cite{Ward1}, they include structures such as lattices of ideals of a ring, lattice ordered groups, algebras of relations. But the study has been revived recently as a study of mathematical structures for substructural logics \cite{Galatos}, that is, logics which lack some of the three structural rules, namely contraction, weakening, and exchange. The variety of residuated lattices is the equivalent algebraic semantics, in the sense of Blok-Pigozzi \cite{Blok}, of the Full Lambek Calculus $FL$. The variety of bounded commutative and integral residuated lattices, that is, the equivalent algebraic semantics of $FL_{ew}$ , the calculus that results from $FL$ by adding two structural rules: exchange and weakening. Important subvarieties of bounded commutative and integral residuated lattices are Boolean algebras, Heyting algebras, $MTL$-algebras, $BL$-algebras, $MV$-algebras.


Topology and algebra, two fundamental domains of mathematics, play complementary roles. When an algebraic structure and a topology come naturally together, the rules that describe the relationship between topology and
algebraic operation are almost always transparent, and naturally the operation has to be continuous or semicontinuous. In recent decades, there has been a lot of research \cite{Borzooei,Borzooei1,Luan,Luan1,Najafi,Yan} on algebras endowed with a topology, the resulting structures are
called topological algebras. The most important topologies on algebras to be considered are linear topologies. For example, it has been noted by G. Birkhoff (\cite{Birkhoof} pp.52-54) that under certain conditions a family of subgroups may be used to definite neighborhoods of the identity in a group and thus introduce a topology. Such
``subgroup topologies" called linear topologies were introduced and their main properties follow from \cite{Dimitric,Fuchs,Hull} treatment of topological algebras as uniform structures. Moreover, Zahiri
and Borzooei \cite{Zahiri} used a special family of filters on a $BL$-algebra $L$ to construct a topology $\tau$ on $L$. This topology
is not only linear but also compatible with all operations of $L$. In \cite{Yang1}, Yang et al. obtained that for each filter $F$ of a residuated lattice $L$, the topology $\mathcal{T}_{F}=\{U\subseteq L: \forall x\in U, x/F\subseteq U\}$ is
not only linear but also Alexandrov, furthermore, $(L,\mathcal{T}_{F})$ is a topological residuated lattice. Also, a question of
interest arises in \cite{Yang1}, that is, whether the sets $\mathcal{F}(L)$ and $LTRL(L)$ are equipotent, where the set $\mathcal{F}(L)$ of all filters of $L$ and the set $LTRL(L)$ of all linear topological residuated lattices with the underlying set $L$? In \cite{Yang}, we prove that $|\mathcal{F}(L)|=|ZLTRL(L)|$ for a residuated lattice $L$, where the set
$ZLTRL(L)$ of all zero-dimensional linear topological residuated lattices with the underlying set $L$. We will correct this result in this paper. Moreover, a perfectly valid question is whether an algebra can in
fact be endowed with a Hausdorff, non-discrete (linear) topology. This question has
been settled in infinite Abelian groups, see Theorem 1.5 in \cite{Dimitric} or
Theorem 7.7 in \cite{Fuchs}. In this paper, we shall discuss the same question for residuated
lattices, that is, the existence of Hausdorff topological residuated lattices.


It is well known that an algebra is profinite if it is representable as a projective limit of finite algebras. Profinite algebras and profinite completions have their origins in algebraic number theory and Galois theory.
Profiniteness is a property referring to the interplay between the topological and the algebraic structure
of a topological algebra, whereas for topological groups, rings, semigroups, and distributive lattices(see for instance \cite{Johnstone}, Sec.VI.2), profiniteness turns out to be a purely topological property as it
is equivalent to the underlying topological space being a Stone space (i.e., the topology is compact, Hausdorff, and zero-dimensional).  N. Nikolov and D. Segal \cite{Nikolov,Nikolov1} have proved that the topological
structure of a finitely generated profinite group is completely determined by its algebraic structure; more precisely, they proved that the subgroups of finite index of such a group are precisely its open subgroups.
This result is generalized to many algebras of logic which
are our main focus, including Heyting
algebras \cite{Bezhanishvili}, $MV$-algebras \cite{Nganou}, orthomodular lattices \cite{Choe}. To the best of our knowledge, up to now, profinite residuated lattices have received very little attention. In this paper, we investigate profinite residuated lattices. In particular, we give characterizations of the profiniteness of compact topological residuated lattices.



\section{Preliminaries}
In this section, we summarize some basic definitions and results, which will be used.


\begin{definition} \label{2.1a}
\emph{\cite{Hohle} An algebraic structure $\mathcal{L}=(L,\wedge,\vee,\odot,\rightarrow,0,1)$ of type $(2,2,2,\\2,0,0)$ is called a \emph{residuated lattice} if it satisfies the following conditions:
\begin{itemize}
 \item[\rm (1)] $(L,\wedge,\vee,0,1)$ is a bounded lattice;
 \item[\rm (2)] $(L,\odot,1)$ is a commutative monoid;
 \item[\rm (3)] $x \odot y \leq z$ if and only if $x \leq y \rightarrow z$, for all $x,y,z \in L$, where $\leq$ is the partial order of the lattice $(L,\wedge,\vee, 0, 1)$.
 \end{itemize}}
\end{definition}

Throughout this paper we will slightly abuse notation $L$ the universe of a residuated lattice $\mathcal{L}=(L,\wedge,\vee,\odot,\rightarrow,0,1)$, when there is no chance to confusion.\\


For convenience of readers, we provide some basic properties of residuated lattices in the following proposition.
\begin{proposition} \label{2.5a}\cite{Turunen}
In any residuated lattice $(L,\wedge,\vee,\odot,\rightarrow,0,1)$, the following properties hold: for any $x, y, z \in L$,
\begin{enumerate}
  \item[\rm$(R_{1})$] $1\rightarrow x=x$, $x \rightarrow 1 = 1$;
  \item[\rm$(R_{2})$] $x \leq y$ if and only if $x \rightarrow y =1$;
  \item[\rm$(R_{3})$] if $x \leq y$, then $y \rightarrow z \leq x \rightarrow z$, $z \rightarrow x \leq z \rightarrow y$ and
              $x \odot z \leq y \odot z$;
  \item[\rm$(R_{4})$] $x \odot (x \rightarrow y) \leq y$;
  \item[\rm$(R_{5})$] $x \odot y \leq x \wedge y$, $x \leq y \rightarrow x$;
  \item[\rm$(R_{6})$] $x \rightarrow (y \rightarrow z)=(x \odot y) \rightarrow z = y \rightarrow (x \rightarrow z)$;
  \item[\rm$(R_{7})$] $x\vee(y\odot z)\geq (x\vee y)\odot(x\vee z)$, hence $x^{m}\vee y^{n}\geq(x\vee y)^{mn}$.
\end{enumerate}
\end{proposition}

\begin{definition} \label{2.6a}\emph{\cite{Turunen} Let $(L,\wedge,\vee,\odot,\rightarrow,0,1)$ be a residuated lattice. A \emph{filter} is a nonempty set $F\subseteq L$ such that for each $x, y\in L$,
\begin{enumerate}
\item[\rm(i)] $x,y\in F$ implies $x\odot y\in F$,
\item[\rm(ii)] if $x\in F$ and $x\leq y$, then $y\in F$.
\end{enumerate}}
\end{definition}

Note that in a residuated lattice $L$, a filter $F$ of $L$ is equivalent to a \emph{deductive system}, that is, $F$ satisfies the following conditions: (i) $1\in F$, and (ii) $x,~x\rightarrow y\in F$ implies $y\in F$.

Let $L$ be a residuated lattice. By $\mathcal{F}(L)$($Con(L)$), we mean the set of all filters (congruences) of $L$. There is close correspondence between congruences and filters of residuated lattices. For each congruence
$\theta$ on the residuated lattice $L$, let $[1]_{\theta}=\{x\in L:(1,x)\in \theta\}$. Then $[1]_{\theta}$ is a filter of $L$, called the filter \emph{determined} by a congruence $\theta$. Conversely, for each filter $F$,
$\theta_{F}=\{(x,y)\in L\times L:(x\rightarrow y)\odot(y\rightarrow x)\in F\}$ is a congruence on $L$, called the congruence \emph{determined} by a filter $F$. For any $x\in L$, let $x/F$ be the equivalence class $x/\theta_{F}$. If we denote by $L/F$ the quotient set $L/\theta_{F}$, then $L/F$ becomes a residuated lattice with the operations induced from those of $L$. Moreover, the following result holds.
\begin{theorem}\label{4.1}\cite{Galatos} For every residuated lattice $L$, we have $$\mathcal{F}(L)\cong Con(L).$$
\end{theorem}

We recall that the notions of \emph{uniform structures} and \emph{uniform topologies}. We shall use the following notation. \\

Let $X$ be a set. We denote by $\Delta_{X}$ the diagonal
in $X\times X$, namely $\Delta_{X}=\{(x,x):x\in X\}\subset X\times X$. Suppose that $R$ is a subset of $X\times X$(in other words, $R$ is a binary relation on $X$). For any $y\in X$, we define the set
$R[y]=\{x\in X:(x,y)\in R\}$. The \emph{inverse} $R^{-1}\subset X\times X$ of $R$ is defined by $R^{-1}=\{(x,y):(y,x)\in R\}$. One says that $R$ is symmetric if it satisfies $R^{-1}=R$. If $R$ and $S$ are subsets of
$X\times X$, we define their \emph{composite} $R\circ S\subset X\times X$ by $R\circ S=\{(x,y): \exists z\in X s.t.~(x,z)\in R~ and~ (z,y)\in S\}$.

\begin{definition}\label{2.5}\cite{Jam} Let $X$ be a set. A \emph{uniform structure} on $X$ is a non-empty set $\mathcal{U}$ of subsets of $X\times X$ satisfying the following conditions:
\begin{itemize}
 \item[\rm (1)] if $V\in \mathcal{U}$, then $\Delta_{X}\subset V$;
 \item[\rm (2)] if $V\in\mathcal{U}$ and $V\subset V'\subset X\times X$, then $V'\in \mathcal{U}$;
 \item[\rm (3)] if $V, W\in\mathcal{U}$, then $V\cap W\in\mathcal{U}$;
 \item[\rm (4)] if $V\in\mathcal{U}$, then $V^{-1}\in\mathcal{U}$;
 \item[\rm (5)] if $V\in \mathcal{U}$, then there exists $W\in \mathcal{U}$ such that $W\circ W\subset V$.
\end{itemize}

\emph{A set $X$ equipped with a uniform structure $\mathcal{U}$ is called a \emph{uniform space} and the elements of $\mathcal{U}$ are called the \emph{entourages} of $X$. Let $X$ be a uniform space. One easily
verifies that it is possible to define a topology on $X$ by taking as open sets the subsets $\Omega\subset X$ which satisfy the following property: for each $x\in\Omega$, there exists an entourage $V\subset X\times X$
such that $V[x]=\Omega$. One says that this topology is the topology \emph{associated} with the uniform structure on $X$. A subset $N\subset X$ is neighborhood of a point $x\in X$ for this topology if and only if
there exists an entourage $V$ such that $N=V[x]$. This topology is Hausdorff if and only if the intersection of the entourages of $X$ coincides with the diagonal $\Delta_{X}\subset X\times X$.}
\end{definition}

\begin{proposition}\label{2.6}\cite{Jam} Let $X$ be a set and let $\mathcal{B}$ be a nonempty set of subsets of $X\times X$. Then $\mathcal{B}$ is a base for some (necessarily unique) uniform structure on $X$ if and only
if it satisfies the following properties:
\begin{itemize}
 \item[\rm (a)] if $V\in \mathcal{B}$, then $\Delta_{X}\subset V$;
 \item[\rm (b)] if $V\in\mathcal{B}$ and $W\in\mathcal{B}$, then there exists $U\in\mathcal{B}$ such that $U\subset V\cap W$;
 \item[\rm (c)] if $V\in\mathcal{U}$, then there exists $W\in\mathcal{B}$ such that $W\subset V^{-1}$;
 \item[\rm (d)] if $V\in\mathcal{U}$, then there exists $W\in\mathcal{B}$ such that $W\circ W\subset V$.
\end{itemize}
\end{proposition}

Given two congruences $\theta_{1},\theta_{2}$ on an algebra $A$. If $\theta_{1}\circ \theta_{2}=\theta_{2}\circ \theta_{1}$,
then we say $\theta_{1}$ and $\theta_{2}$ are \emph{permutable}, or $\theta_{1}$ and $\theta_{2}$ \emph{permute}. An algebra $A$ is \emph{congruence permutable} if
every pair of congruences on $A$ permutes. A class $\mathcal{K}$ of algebras is \emph{congruence permutable} iff every algebra in $\mathcal{K}$ has congruence permutable, see \cite{Burris}.\\

In what follows, we give a general result about topological algebras.\\

A \emph{topological (universal) algebra} is an object $\mathbf{A}=(A,(f_{\alpha})_{\alpha\in I},\tau)$ where $A$ is a set, $(f_{\alpha})_{\alpha\in I}$ a family of maps $f_{\alpha}:X^{n_{\alpha}}\rightarrow X$, $n_{\alpha}$
the
\emph{arity} of the operation $f_{\alpha}$, and $\tau$ a topology on $A$ such that each $f_{\alpha}$ maps the product space $(A,\tau)^{n_{\alpha}}$ continuously to the space $(A,\tau)$. $(A,(f_{\alpha})_{\alpha\in I})$ is
called the \emph{underlying algebra}, and $(X,\tau)$ the \emph{underlying space}.\\

By Proposition \ref{2.6}, we have the following result, which is of great significance to study the topological algebras.

\begin{theorem}For a congruence permutable algebra $\mathbf{A}=(A,(f_{\alpha})_{\alpha\in I})$, if $(\mathscr{C},\subseteq)$ is a down-directed set, then $(A,(f_{\alpha})_{\alpha\in I},\tau)$ is a topological
algebra, where $\mathscr{C}\subseteq Con A$ and $\tau$ is a topology induced by $\mathscr{C}$.
\end{theorem}

Clearly, $\mathscr{C}$ satisfies the conditions ($a$) and ($c$) of Proposition \ref{2.6}. The hypothesis that $(\mathscr{C},\subseteq)$ is a down-directed set corresponds to the condition ($b$) of Proposition \ref{2.6}, and the
congruence permutability implies the condition ($d$) of Proposition \ref{2.6}. Thus, by Proposition \ref{2.6}, $\mathscr{C}$ is a base for the uniform structure on $A$ and $\tau$ is a topology associated with this uniform
structure on $A$. Note that $$f_{\alpha}([a]_{\theta},[b]_{\theta})\subseteq [f_{\alpha}(a,b)]_{\theta},$$
for any $a,b\in A$ and $\theta\in\mathscr{C}$. It follows that $(A,(f_{\alpha})_{\alpha\in I},\tau_{\theta})$ is a topological algebra, where $$\tau_{\theta}=\{U\subseteq A: \forall a\in U, [a]_{\theta}\subseteq U\}.$$
Since $\tau=\sup\{\tau_{\theta}:\theta\in\mathscr{C}\}$, we deduce that $(A,(f_{\alpha})_{\alpha\in I},\tau)$ is a topological algebra. \\

The variety $\mathcal{RL}$ of residuated lattices is arithmetical. It has the congruence extension property (CEP),
and is congruence 1-regular, i.e., for any congruence $\theta$, the coset of 1 uniquely determines
$\theta$. Furthermore, the congruences of a residuated lattice are completely
determined by filters (see Theorem 2.4). Thus, in \cite{Yang1}, we studied topological residuated lattices induced by a system of filters. Since the variety $\mathcal{RL}$ is congruence permutable,
it follows from Theorem 2.7 that a residuated lattice endowed with the topology induced by a system of filters is a topological residuated lattice, see Theorem 3.13 in \cite{Yang1}.\\

A topological space $(X,\mathcal{T})$ is a \emph{zero-dimensional space} if $\mathcal{T}$ has a clopen base.
Let $\tau$ and $\tau'$ be two topologies on a given set $X$. If $\tau\subseteq\tau'$, then we say that $\tau'$ \emph{finer than} $\tau$.

\begin{definition}\label{2.1}\emph{\cite{Yang1} Let $\mathcal{F}$ be a family of filters of a residuated lattice $L$. Then $\mathcal{F}$ called a \emph{system of filters} of $L$ if $(\mathcal{F},\subseteq)$ is
a down-directed set.}
\end{definition}

 A topology $\tau$ of a residuated lattice $L$ is called a \emph{linear topology} on $L$ if there exists a base $\beta$ for $\tau$ such that $B$ is a filter of $L$, for any element $B$ of $\beta$ containing 1.

If $(\mathcal{F},\subseteq)$ is a system of filters of a residuated lattice $L$, then, by Theorem 2.7 or Theorem 3.13 in \cite{Yang1}, $$\mathcal{T}_{\mathcal{F}}=\{U\subseteq L:\forall x\in U,\exists F\in\mathcal{F}~s.t.,x/F\subseteq U\}$$
is a linear topology on $L$. We call it a \emph{linear topology induced by $\mathcal{F}$}. Also, $(L,\mathcal{T}_{\mathcal{F}})$ is a topological residuated lattice.


\section{Profinite residuated lattices}

In this section, we study profinite residuated lattices and give several characterizations of profiniteness in compact topological residuated lattices.\\

Let $I=(I,\leq)$ denote a \emph{directed partially set} or \emph{directed poset}, that is, $I$ is a set with a binary relation $\leq$ such that $(I,\leq)$ is a poset and if $i,j\in I$, there exists some $k\in I$ such that
$i,j\leq k$.
\begin{definition}\label{3.1}\emph{\cite{Gratzer1}
(i) By an \emph{inverse or projective system} in a category $\mathcal{D}$ we mean a family $\{B_{i},\pi_{ij},I\}$ of objects, indexed by a directed poset $I$, with a family of morphisms $\pi_{ij}: B_{i}\rightarrow B_{j}$, for any $j\leq i$, satisfying the following conditions:
\begin{itemize}
\item[\rm (1)] $\pi_{ik}=\pi_{jk}\circ \pi_{ij}$, for any $k\leq j\leq i$;
\item[\rm (2)] $\pi_{ii}=id_{B_{i}}$, for any $i\in I$.
\end{itemize}
For brevity we say that $\{B_{i},\pi_{ij},I\}$ is an inverse system in $\mathcal{D}$.}

\emph{(ii) The \emph{inverse limit} of an inverse system $\{B_{i},\pi_{ij},I\}$ in a category $\mathcal{D}$ is an object $B$ of $\mathcal{D}$ together with a family $\{\phi_{i}:B\rightarrow B_{i}\}_{i\in I}$
of morphisms (which is often denoted by $\{B,\phi_{i}\}_{i\in I}$) satisfying the conditions:
\begin{itemize}
\item[\rm (1)] $\pi_{ij}\circ\phi_{i}=\phi_{j}$, for any $i,j\in I, j\leq i$(this condition often is called \emph{compatible condition});
\item[\rm (2)] for any object $B'$ of $\mathcal{D}$, together with a family of morphisms $\lambda_{i}: B'\rightarrow B, i\in I$, if $\pi_{ij}\circ\lambda_{i}=\lambda_{j}$, for any $i,j\in I, j\leq i$ then there exists a unique morphism $\lambda:B'\rightarrow B$ such that $\phi_{i}\circ\lambda=\lambda_{i}$, for any $i\in I$.
 \end{itemize}
    The inverse limit of the above system is denoted by $\lim\limits_{\longleftarrow}B_{i}$.}
\end{definition}

Recall from Gr\"atzer \cite{Gratzer1} that the inverse limits of families of algebras are constructed in the following way.
\begin{theorem}\cite{Gratzer1}\label{3.2}
 Let $\{B_{i},\pi_{ij},I\}$ be an inverse system of same algebras, $\prod_{i\in I}B_{i}$ be its product and $\pi_{i}:\prod_{i\in I}B_{i}\rightarrow B_{i}$ is defined by $\pi_{i}((x_{i})_{i\in I})=x_{i}$ for any $i\in I$. Let $$B=\{(b_{i})_{i\in I}\in\prod_{i\in I}B_{i}:\pi_{ij}(b_{i})=b_{j},j\leq i\}.$$
 Then B is a subalgebra of $\prod_{i\in I}B_{i}$ and $\{B,\phi_{i}\}_{i\in I}$ is the inverse limit of $\{B_{i},\pi_{ij},I\}$, where $\phi_{i}=\pi_{i}\upharpoonright_{B}$ for any $i\in I$.
\end{theorem}

Let $A$ be an algebra. Let $I$ denote the set of all congruences $\theta$ on $A$ such that $A/\theta$ is finite. We denote the image of $a\in A$ in $A/\theta$ by $[a]_{\theta}$. If $\theta\subseteq\theta'$, then there is a canonical projection $\varphi_{\theta'\theta}:A/\theta'\rightarrow A/\theta$ given by $\varphi_{\theta'\theta}([a]_{\theta'})=[a]_{\theta}$. Then $(I,\supseteq)$ is a directed set, and
$(\{A/\theta\},\{\varphi_{\theta'\theta}\},I)$ is an inverse system of algebras. The \emph{profinite completion $\widehat{A}$} of $A$ is the inverse limit of $(\{A/\theta\},\{\varphi_{\theta'\theta}\},I)$. By \cite{Gratzer1}, we may identify $\widehat{A}$ with the subalgebra of $\prod_{\theta\in I}A/\theta$ consisting of all $([a]_{\theta})_{\theta\in I}$ for which $\varphi_{\theta'\theta}([a]_{\theta'})=[a]_{\theta}$ whenever
$\theta'\subseteq\theta$. We define the canonical homomorphism $e_{A}:A\rightarrow \widehat{A}$ by $e_{A}(a)=([a]_{\theta})_{\theta\in I}$. Let $\pi_{\theta'}:\prod_{\theta\in I}A/\theta\rightarrow A/\theta'$ denote the projection map. We also denote the restriction of $\pi_{\theta'}$ to $\widehat{A}$ by $\pi_{\theta'}$.
One would expect a completion of an algebra $A$ to be an extension of $A$, however, the canonical homomorphism $e_{A}:A\rightarrow \widehat{A}$ is not always an embedding.

Let $(I,\leq)$ be a directed poset. Assume that $I'$ is a subset of $I$ in such a way that $(I',\leq)$ becomes a directed poset. We say that $I'$ is \emph{cofinal} in $I$ if for every $i\in I$ there is some $i'\in I'$ such that $i\leq i'$. If $\{X_{i},\varphi_{ij}, I\}$ is an inverse system and $I'$ is cofinal in $I$, then $\{X_{i},\varphi_{ij}, I'\}$ becomes an inverse system in an obvious way, and we say that $\{X_{i}, \varphi_{ij},I'\}$
is a cofinal subsystem of $\{X_{i},\varphi_{ij},I\}$.
\begin{proposition}\label{3.3}
Let $\{A_{i},\varphi_{ij},I\}$ be an inverse system of same algebras over a directed poset $I$ and assume that $I'$ is a cofinal subset of $I$. Then $$\varprojlim_{i\in I}X_{i}\cong\varprojlim_{i'\in I'}X_{i'}.$$
\end{proposition}

\begin{proof}
Assume that $\{A_{i},\varphi_{ij},I'\}$ is a cofinal subsystem of $\{A_{i},\varphi_{ij},I\}$ and denote by $(\varprojlim_{i'\in I'},\varphi_{i'}')$ and $(\varprojlim_{i\in I},\varphi_{i})$ their corresponding inverse limits, respectively. For any $j\in I$, let $j'\in I'$ be such that $j\leq j'$. Define $$\overline{\varphi}_{j}:\varprojlim_{i'\in I'}X_{i'}\rightarrow X_{j}$$
as the composition of canonical mappings $\varphi_{j'j},\varphi_{j'}'$. Observe that the maps $\overline{\varphi}_{j}$ are well-defined and compatible. Hence they induce a map
$$\overline{\varphi}:\varprojlim_{i'\in I'}X_{i'}\rightarrow \varprojlim_{i\in I}X_{i}$$
such that $\varphi_{j}\overline{\varphi}=\overline{\varphi}_{j}(j\in I)$. We claim that the mapping $\overline{\varphi}$ is a bijection. Noth that if $(x_{i'})\in \varprojlim_{i'\in I'}X_{i'}$ and
$\overline{\varphi}(x_{i'})=(y_{i})$, then $y_{i'}=x_{i'}$ for $i'\in I'$. Since $I'$ is cofinal in $I$, thus $\overline{\varphi}$ is an injection. To see that $\overline{\varphi}$ is a surjection, let $(y_{i})\in \varprojlim_{i\in I}X_{i}$ and consider the element $(x_{i'})$, where $x_{i'}=y_{i'}$ for every $i'\in I'$. Then $(x_{i'})\in \varprojlim_{i'\in I'}X_{i'}$ and obviously, $\overline{\varphi}(x_{i'})=(y_{i})$.
This proves the claim. The rest of the proof is obvious.
\end{proof}

In what follows we shall be specially interested in the topological space $X$ that arise as inverse limits
\[X=\lim_{\substack{\longleftarrow\\ i\in I}}X_{i}\]
of finite spaces $X_{i}$ endowed with the discrete topology. We call such a space \emph{profinite space} or a \emph{Boolean space}. A topological space is \emph{totally disconnected} if every point in the space is its own connected component.

\begin{theorem}\label{3.4}\cite{Ribes}
Let $X$ be a topological space. Then the following conditions are equivalent.
\begin{enumerate}
  \item[(a)] $X$ is a profinite space;
  \item[(b)] $X$ is compact Hausdorff and totally disconnected;
  \item[(c)] $X$ is compact Hausdorff and admits a base of clopen sets for its topology.
\end{enumerate}
\end{theorem}

\begin{definition}\label{3.5}
\emph{We call a residuated lattice $L$ \emph{profinite} if it is isomorphic to the inverse limit of an inverse system of finite residuated lattices. The class of this algebras denoted by $\mathbf{Pro}\mathcal{RL}$.}
\end{definition}

Thus, we see that if we are given a profinite algebra $A\cong\varprojlim_{i\in I}A_{i}$, then $A$ is a topological algebra in its profinite topology, which can definite in a natural way using the limiting cone
$(\pi_{i}:A\rightarrow A_{i})_{I}$. We will now record the fundamental fact that every profinite topology is Boolean, i.e., every profinite topology is compact Hausdorff and zero-dimensional.
\begin{proposition}\label{3.6}\cite{Banaschewski}
If $A\cong\varprojlim_{i\in I}A_{i}$ is a profinite algebra, then $A$ is a Boolean topological algebra in its profinite topology.
\end{proposition}

\begin{remark}\label{3.7}
We have chosen to first describe profinite algebras using universal algebra, via. as the limit of a diagram of finite algebras $\{A_{i},\varphi_{ij},I\}$, and then to indicate that one can definite a topology on
$\varprojlim_{i\in I}A_{i}$, making $\varprojlim_{i\in I}A_{i}$ a topological algebra. Alternatively, one can start by endowing each algebra $A_{i}$ with the discrete topology, and then show that one can also take limits of topological algebras, so that it follows immediately that $\varprojlim_{i\in I}A_{i}$ is a topological algebra. One can then show that $\varprojlim_{i\in I}A_{i}$ is a closed subalgebra of $\prod_{i\in I}A_{i}$, so that it
follows from general topology that the profinite topology on $\varprojlim_{i\in I}A_{i}$ is a Boolean topology.
\end{remark}

We begin this section by showing that an infinite profinite residuated lattice cannot be countable. This is a general fact for locally compact topological groups \cite{Fuchs}, but here we present a proof for profinite residuated lattices only. The first part of the following proposition is similar to the classical Baire category theorem, valid for locally compact spaces.
\begin{proposition} \label{3.8}
Let $L$ be a profinite residuated lattice.
\begin{enumerate}
  \item[(a)] Let $C_{1},C_{2},\ldots$ be countably infinite set of nonempty closets of $L$ having empty interior. Then
  $$L\neq\bigcup_{n=1}^{\infty}C_{i}.$$
  \item[(b)] The cardinality $|L|$ of $L$ is either finite or uncountable.
\end{enumerate}
\end{proposition}

\begin{proof}
Part $(b)$ follows immediately from $(a)$. To prove $(a)$, assume that $L=\bigcup_{n=1}^{\infty}C_{i}$, where each $C_{i}$ is a nonempty closed subset of $L$ with empty interior. Then $D_{i}=L-C_{i}$ is a dense open subset of $L$, for each $i\in\{1,2,\ldots\}$.

Next consider a nonempty open subset $U_{0}$ of $L$; then $U_{0}\cap D_{1}$ is open and nonempty since $D_{1}$ is open and dense in $L$. By Theorem \ref{3.4}, there is a nonempty clopen subset $U_{1}$ of $U_{0}\cap D_{1}$.
 Similarly, $U_{1}\cap D_{2}$ is open and nonempty, therefore there is a nonempty clopen subset $U_{2}$ of $U_{1}\cap D_{2}$. Proceeding in this manner we obtain a nested sequence of clopen nonempty subsets
  $$U_{1}\supseteq U_{2}\supseteq\cdots\supseteq U_{i}\supseteq \cdots$$
such that $U_{i}\subseteq D_{i}\cap U_{i-1}$ for each $i\in\{1, 2,\ldots\}$. Since $L$ is compact and the closed sets $U_{i}$ have the finite intersection property, we have that
    $$\bigcap_{i=1}^{\infty}U_{i}\neq\emptyset.$$
On the other hand,
$$\bigcap_{i=1}^{\infty}U_{i}\subseteq\bigcap_{i=1}^{\infty}D_{i}=L-(\bigcup_{i=1}^{\infty}C_{i})=\emptyset,$$
a contradiction.
\end{proof}

\begin{lemma}\label{3.8a}
Let $\{L_{i},\varphi_{ij},I\}$ be an inverse system of topological residuated lattices $L_{i}$. Then we have:
\begin{enumerate}
  \item[(i)] the inverse limit $\underleftarrow{\lim}\{L_{i}:i\in I\}$ forms a subalgebra of $\prod\{L_{i}:i\in I\}$;
  \item[(ii)] if each topological residuated lattice in $\{L_{i}:i\in I\}$ is Hausdorff, then the inverse limit $\underleftarrow{\lim}\{L_{i}:i\in I\}$ is topologically closed in the product $\prod\{L_{i}:i\in I\}$;
  \item[(iii)] if each topological residuated lattice in $\{L_{i}:i\in I\}$ is non-empty, compact and Hausdorff, then the inverse limit $\underleftarrow{\lim}\{L_{i}:i\in I\}$ is non-empty.
\end{enumerate}
\end{lemma}

\begin{proof}
The proof of $(i)$ is easy. We sketch the proof of $(ii)$ and $(iii)$. First assume that each $L_{i}$ is Hausdorff. For each $k\in I$, define the set
$$A_{k}=\{a\in\prod_{i\in I}:(\forall i\leq k)a_{i}=\varphi_{ki}(a_{k})\}.$$
Each $A_{k}$ is topological closed in $\Pi_{i\in I}L_{i}$, and $\underleftarrow{\lim}\{L_{i}:i\in I\}=\bigcap_{k\in I}A_{k}$. So the inverse limit is topological closed. Now assume further that each $L_{i}$ is non-empty
and compact. Since $I$ is direct and each $L_{i}$ is non-empty, every finite intersection of members of $\{A_{k}:k\in I\}$ is non-empty. By compactness, the set $\bigcap_{k\in I}A_{k}$ is non-empty.
\end{proof}

\begin{proposition}\label{3.9}
Let $$L=\varprojlim_{i\in I}L_{i},$$
where $\{L_{i},\varphi_{ij},I\}$ is an inverse system of finite residuated lattices $L_{i}$, and let $$\pi_{i}:L\rightarrow L_{i}(i\in I)$$
be the projection homomorphisms. Then $$\{S_{i}:S_{i}=Ker(\pi_{i})\}$$
is a fundamental system of open neighborhoods of 1 in L.
\end{proposition}

\begin{proof}
Consider the family of neighborhoods of 1 in $\prod_{i\in I}L_{i}$ of the form $$(\prod_{i\neq i_{1},\ldots,i_{t}}L_{i})\times\{1\}_{i_{1}}\times\cdots\times\{1\}_{i_{t}},$$
for any finite collection of indexes $i_{i},\ldots,i_{t}\in I$, where $\{1\}_{i}$ denotes the subset of $L_{i}$ consisting of the identity element. Since each $L_{i}$ is discrete, this family is fundamental system of
neighborhoods of the identity element of $\prod_{i\in I}L_{i}$. Let $i_{0}\in I$ be such that $i_{1},\ldots,i_{t}\leq i_{0}$. Then
$$L\cap[(\prod_{i\neq i_{0}})L_{i}\times\{1\}_{i_{0}}]=L\cap[(\prod_{i\neq i_{1},\ldots,i_{t}}L_{i})\times\{1\}_{i_{1}}\times\cdots\times\{1\}_{i_{t}}].$$
Therefore the family of neighborhoods of 1 in $L$, of the form $$L\cap[(\prod_{i\neq i_{0}})L_{i}\times\{1\}_{i_{0}}]$$
is a fundamental system of open neighborhoods of 1. Finally, observe that $$L\cap[(\prod_{i\neq i_{0}})L_{i}\times\{1\}_{i_{0}}]=Ker(\pi_{i_{0}})=S_{i_{0}}.$$
\end{proof}

\begin{theorem}\label{3.10}
Let $(L,\tau)$ be a compact topological residuated lattice. Then the following conditions are equivalent:
\begin{enumerate}
  \item[(i)] $(L,\tau)$ is profinite;
  \item[(ii)] there is a set $\mathcal{S}$ of clopen filters(congruences) on $L$ such that
  \item[(a)] the set $(\mathcal{S},\supseteq)$ is directed, and
  \item[(b)] for all $x,y\in L$ with $x\neq y$, there is a filter $F\in\mathcal{S}$ such that $x/F\neq y/F$ (congruence $\theta\in\mathcal{S}$ such that $[x]_{\theta}\neq [y]_{\theta}$).
\end{enumerate}
\end{theorem}

\begin{proof}
First assume that $(L,\tau)$ is compact, and let $\mathcal{S}$ be a set of clopen filters on $L$ such that $(a)$ and $(b)$ hold. We can set up an inverse system on $\{L/F:F\in\mathcal{S}\}$, where $\mathcal{S}$ is ordered
by reverse inclusion: for all $F\subseteq G$ in $\mathcal{S}$, define the homomorphism $\varphi_{FG}:L/F\rightarrow L/G$ by $\varphi_{FG}(x/F)=x/G$. Since $(L,\tau)$ is compact, each quotient $L/F$ is finite. It is
straightforward to prove that we can definite an isomorphism $\psi:L\rightarrow \underleftarrow{\lim}\{L/F:F\in\mathcal{S}\}$ by $\psi(x)(F)=x/F$. (To prove that $\psi$ is surjective, use the compactness of $(L,\tau)$) to
find a member of $\bigcap\{z(F):F\in\mathcal{S}\}$, for each $z\in\underleftarrow{\lim}\{L/F:F\in\mathcal{S}\}$.) This inverse limit is clearly surjective. By condition $(b)$, it follows that $\psi$ is injective.

Now assume that $(L,\tau)$ is profinite. We can assume that it is equal to an inverse limit $\underleftarrow{\lim}\{L_{i}:i\in I\}$ of finite residuated lattices, that is, $L=\underleftarrow{\lim}\{L_{i}:i\in I\}$.
Let $\mathcal{S}=\{Ker(\pi_{i}):i\in I\}$. Then $Ker(\pi_{i})$ is a clopen filters of $L$, for each $i\in I$. Thus the condition $(a)$ clearly holds. By Proposition \ref{3.9} and profinite topological is always Hausdorff,
it follows that $\bigcap\{Ker(\pi_{i}):i\in I\}=\{1\}$. Thus, it deduces that the condition $(b)$ holds.
\end{proof}

The following result is a particular case of a theorem of universal algebra.
\begin{proposition} \label{4.6}
A residuated lattice $L$ is a subdirect product of a family $L_{i}$ of residuated lattices if and only if there is a family $\{F_{i}\}_{i\in I}$ of filters of $L$ such that
\begin{itemize}
\item[\rm (i)] $L_{i}\cong L/F_{i}$ for each $i\in I$;
\item[\rm (ii)] $\bigcap_{i\in I}F_{i}=\{1\}$.
\end{itemize}
\end{proposition}

Recall that a topological algebra $(L,\tau)$ is \emph{residually finite} \cite{Schneider} if, for any two distinct elements $x,y\in L$, there exists a finite discrete topological algebra $F$ as well as
continuous homomorphism $\varphi: L\rightarrow F$ such that $\varphi(x)\neq\varphi(y)$.
\begin{theorem}\label{3.11}
Let $(L,\tau)$ be a compact topological residuated lattice. Then the following conditions are equivalent:
\begin{enumerate}
  \item[(i)] $(L,\tau)$ is residually finite;
  \item[(ii)] $(L,\tau)$ is profinite;
  \item[(iii)] $(L,\tau)$ is a closed subdirect product of finite discrete residuated lattices.
\end{enumerate}
\end{theorem}

\begin{proof}
First, let us prove that $(i)$ implies $(iii)$. For $x,y\in L$ with $x\neq y$. Then there exist a finite residuated lattice $F$ and a homomorphism $\varphi_{x,y}:L\rightarrow F$ such that
$\varphi_{x,y}(x)\neq \varphi_{x,y}(y)$. It follows that $Ker(\varphi_{x,y})$ is a finite index clopen filter of $(L,\tau)$, for each $x,y\in L$ with $x\neq y$. It is easy to check that $L$ is a subdiect product of
$\{L/Ker(\varphi_{x,y}):x,y\in L ~and~ x\neq y\}$. By Lemma \ref{3.8a} $(ii)$, it deduces that $(iii)$ holds. Now assume that the condition $(iii)$ holds. Let $L$ be a closed subdirect product of finite residuated lattices
$\{L_{i}: i\in I\}$. By Proposition \ref{4.6}, there exists a family $\{F_{i}:i\in I\}$ of filters of $L$ such that $\bigcap_{i\in I}F_{i}=\{1\}$ and $L_{i}\cong L/F_{i}$ for each $i\in I$. Thus $(L,\tau)$ is closed
in $\prod_{i\in I}L/F_{i}$, and hence $(L,\tau)$ is residually finite.

Since $(iii)$ implies $(ii)$ is obvious, it suffices to show the converse. Let $L$ be a topological closed subalgebra of a product
$\Pi_{i\in I}L_{i}$, where $L_{i}$ is a finite residuated lattice for each $i\in I$. Let $\mathcal{S}$ be the set of all finite non-empty subsets of $I$, ordered by inclusion. For each $J\in\mathcal{S}$, let
$L\upharpoonright_{J}:L\rightarrow \prod_{j\in J}L_{j}$ be the projection. For all $K\supseteq J$ in $\mathcal{S}$, let $\varphi_{KJ}:L\upharpoonright_{K}\rightarrow L\upharpoonright_{J}$ be the projection. Then
$\{L\upharpoonright_{J}:J\in\mathcal{S}\}$ forms an inverse system of topological residuated lattices. Define the inverse limit $Y=\underleftarrow{\lim}\{L\upharpoonright_{J}:J\in\mathcal{S}\}$. We want to define an
isomorphism $\psi:L\rightarrow Y$ by $\psi(x)(J)=x\upharpoonright_{J}$. For each $x\in L$, we have $\psi(x)\in Y$, since $\varphi_{KJ}((\psi(x))(K))=\varphi_{KJ}(x\upharpoonright_{K})=x\upharpoonright_{J}=\psi(x)(J)$,
for all $K\supset J$ in $\mathcal{S}$. The map $\psi$ is a homomorphism, since the projections are homomorphisms. It is easy to check that $\psi$ is injective. To see that $\psi$ is continuous, consider a subbasic open
set $U=\{y\in Y:y(Y)(j)\in V\}$ of $Y$, where $j\in J\in\mathcal{S}$ and $V$ is open in $L_{j}$. We have $\psi^{-1}(U)=\{x\in L:x(j)\in V\}$, which is open in $L$. It remains to verify that $\psi$ is surjective. Let
$y\in Y$. For each $J\in \mathcal{S}$, define the closed set $V_{J}=\{x\in L:x\upharpoonright_{J}=y\upharpoonright_{J}\}$ in $L$. Since $y\in Y$, we know that $y(J)\in L\upharpoonright_{J}$ and
therefore $V_{J}\neq\emptyset$, for each $J\in\mathcal{S}$. The set $\{V_{J}:J\in\mathcal{S}\}$ is closed under the finite intersection, since
$V_{J_{1}}\cap V_{J_{2}}\cap \cdots \cap V_{J_{n}}=V_{J_{1}\cup \cdots \cup J_{n}}\neq\emptyset$. So, as $L$ is compact, we can choose $x\in\bigcap\{V_{J}:J\in\mathcal{S}\}$. For each $J\in \mathcal{S}$, we have
$\psi(x)(J)=x\upharpoonright_{J}=y(J)$. Thus $\psi(x)=y$.
\end{proof}


\section{Finiteness conditions in residuated lattices}

The elementary theory of vector spaces concerns itself with spaces of finite dimension, in such spaces there does not exist an infinite strictly ascending or strictly descending chain of subspaces. Similarly, the elementary theory of groups concerns itself with finite groups, or at any rate with groups which are not ``too infinite" in some sense. The purposes of this section is the discussion of various finiteness conditions which can be imposed on a residuated lattice.\\

So far we have considered a quite arbitrary residuated lattice. To go further, however, and obtain deeper theorems we need to impose some conditions. The most convenient way is in the form of ``chain conditions".
\begin{definition}\label{4.2} \emph{Let $L$ be a residuated lattice. We call $L$ satisfying the \emph{descending chain condition }($DCC$ for short), if every strictly descending chain of filters}
$$\cdots \subset F_{2}\subset F_{1} $$
\emph{is finite.}

\emph{An equivalent formulation is: if}
$$\cdots \subset F_{2}\subset F_{1} $$
\emph{is an infinite descending sequence of filters, then there exists an integer $n$ such that}
$$F_{i}=F_{n},for~ i=n+1,n+2\cdots.$$
\end{definition}

\begin{definition}\label{4.3}
\emph{A residuated lattice is said to satisfy the \emph{minimal condition} if every non-empty collection $\mathcal{C}$ of filters has a minimal element, that is, if there exists a filter in $\mathcal{C}$ which is not contained in any other filter in the collection $\mathcal{C}$.}
\end{definition}

\begin{proposition}\label{4.4}
Let $L$ be a residuated lattice. Then the following assertions are equivalent:
\begin{itemize}
\item[\rm (1)] $L$ satisfies the descending chain condition;
\item[\rm (2)] $L$ satisfies the minimal condition;
\item[\rm (3)] each system of filters of $L$ has a minimum element.
\end{itemize}
\end{proposition}

\begin{proof}
The equivalence of conditions (1) and (2) is the standard method. Now, suppose that the condition (2) is true. Let $(\mathcal{F},\subseteq)$ be a system of filters. It must have a minimal filter $F$ in $\mathcal{F}$. For any another filter $G$ in $\mathcal{F}$, since $(\mathcal{F},\subseteq)$ is down-directed, we have $F\subseteq G$. Thus $F$ is a minimum filter in $(\mathcal{F},\subseteq)$. It is valid that condition (3) implies condition (1) , since every chain of filters is a system of filters.
\end{proof}




\begin{theorem}\label{4.7}
Let $L$ be a residuated lattice and $\mathcal{F}=\{F_{i}: i\in I\}$ be a system of filters of $L$. Then the following statements are equivalent:
\begin{itemize}
\item[\rm (i)] $(L, \mathcal{T}_{F})$ is a Hausdorff space;
\item[\rm (ii)] $(L, \mathcal{T}_{F})$ is a $T_{1}$-space;
\item[\rm (iii)] $(L, \mathcal{T}_{F})$ is a $T_{0}$-space;
\item[\rm (iv)] $\bigcap_{i\in I}F_{i}=\{1\}$;
\item[\rm (v)] $L$ is a subdirect product of a family $L/F_{i}$ of residuated lattices.
\end{itemize}
\end{theorem}

\begin{proof}
The equivalences of (i)-(iv) are similar to Proposition 3.20 in \cite{Yang2}. The equivalence of (iv) and (v) follows from Proposition \ref{4.6}.
\end{proof}

A filter $F$ of a residuated lattice $L$ is said to be \emph{prime} if $x\vee y\in F$ implies $x\in F$ or $y\in F$ for all $x,y\in L$, see \cite{Kondo}. We denote the set of all prime filters of a residuated lattice $L$ by $\mathcal{P}(L)$.

\begin{theorem}\label{4.8} \cite{Hohle}
Let $L$ be a residuated lattice. For every element $a\in L$ with $a\neq 1$ there is a prime filter $P$ in $L$ with $a\notin P$.
\end{theorem}

\begin{corollary}\label{4.9}\cite{Hohle}
Every proper filter of a residuated lattice is an intersection of prime filters.
\end{corollary}

\begin{corollary}\label{4.10}\cite{Hohle}
Let $L$ be a residuated lattice. Then $\bigcap\mathcal{P}(L)=\{1\}$.
\end{corollary}

Given a system $\mathcal{F}=\{F_{i}: i\in I\}$ of filters of a residuated lattice $L$. It is easily checked that the following assertions hold:
\begin{itemize}
\item[\rm (i)] $(L, \mathcal{T}_{F})$ is a discrete space if and only if $\{1\}\in \mathcal{F}$;
\item[\rm (ii)] $(L, \mathcal{T}_{F})$ is an anti-discrete space if and only if $\{L\}=\mathcal{F}$.
\end{itemize}



We call $\mathcal{F}(L)\backslash\{1\}$ is a \emph{global system of filters} if it is a system of filters, i.e., $(\mathcal{F}(L)\backslash\{1\},\subseteq)$ is closed under finite intersections. Recall that a congruence $\theta$ on an algebra is a \emph{factor congruence} if there is a congruence $\theta^{*}$ such that $\theta\cap\theta^{*}=\triangle$, $\theta\cup\theta^{*}=\nabla$ and $\theta$ permutes with $\theta^{*}$, see \cite{Burris} for details. An algebra is \emph{directly indecomposable} if it is not isomorphic to a direct product of two nontrivial algebras.
\begin{proposition}\label{4.11}
Let $L$ be a residuated lattice. Then the following conditions are equivalent:
\begin{enumerate}
  \item[(1)] $\mathcal{F}(L)\backslash\{1\}$ is a global system of filters;
  \item[(2)] any finite intersections of $\mathcal{F}(L)\backslash\{1\}$ is not equal to $\{1\}$;
  \item[(3)] there is no non-trivial factor congruence on $L$;
  \item[(4)] $L$ is directly indecomposable.
\end{enumerate}
\end{proposition}

\begin{proof}
It is easy to see that (1),(2) and (3) are equivalent. According to Corollary 7.7 in \cite{Burris}, the conditions (3) and (4) are equivalent.
\end{proof}

An algebra $A$ is \emph{subdirectly irreducible} iff $A$ is trivial or there is a minimum congruence in $Con(A)-\Delta$, see \cite{Burris}.\\

From Theorem 8.5 in \cite{Burris}, we know that a subdirectly irreducible residuated lattice is directly indecomposable. But the converse, in general, is not true. In what follows, we give some conditions the converse is also true.

\begin{lemma} \label{4.12}
Let $L$ be a directly indecomposable residuated lattice.  If $L$ is subdirectly irreducible, then the corresponding zero-dimensional linear topological residuated lattice of its global system of filters is non discrete.
\end{lemma}

\begin{proof}Suppose that $L$ is subdirectly irreducible. By Theorem 8.4 in \cite{Burris}, the global system of filters $\mathcal{F}(L)\backslash\{1\}$ has a minimum filter $F_{0}$. Thus the zero-dimensional linear topology of its global system of filters is equal to $\mathcal{T}_{F_{0}}$, and it obviously is non-trivial.
\end{proof}

\begin{theorem}\label{4.13} Let $L$ be a directly indecomposable residuated lattice and $L$ satisfy $DCC$. Then $L$ is subdirectly irreducible iff the corresponding zero-dimensional linear topological residuated lattice of its global system of filters is non discrete.
\end{theorem}

\begin{proof}
The ``only if'' part is Lemma \ref{4.12}. To prove the ``if'' part, by Proposition \ref{4.4}, the global system of filters of $L$ has a minimum element $F_{0}$. Therefore, by Theorem 8.4 in \cite{Burris}, we obatin that $L$ is subdirectly irreducible.
\end{proof}



\begin{proposition}\label{5.6} If $L$ is a residuated lattice that satisfies the descending chain condition then every element of $\mathcal{F}(L)\setminus\{1\}$ can be expressed uniquely as an irredundant join of join-irreducible filters.
\end{proposition}

\begin{proof}
Note that the lattice $(\mathcal{F}(L),\subseteq)$ is a distributive lattice. Thus, it directly follows from Theorem 5.2 in \cite{Blyth}.
\end{proof}

In the power set lattice $\mathcal {P}(X)$, the join-irreducible elements are exactly the singleton subsets. Let $\mathcal{JF}(L)$ denote the set of join-irreducible filters of a residuated lattice $L$.
Thus generally the set $\mathcal{JF}(L)$ of all join-irreducible filters of $L$ is not a down
directed set under the set inclusion. In the following we show that $(\mathcal{JF}(L)\cup \{1\},\supseteq)$ is a directed set under some finiteness conditions.
\begin{lemma}\label{5.10}
If $L$ is a residuated lattice that satisfies the descending chain condition, then $(\mathcal{JF}(L)\cup\{1\},\supseteq)$ is cofinal in $(\mathcal{F}(L),\supseteq)$.
\end{lemma}

\begin{proof}
Suppose that $F,G\in \mathcal{JF}(L)$. If $F\cap G=\{1\}$, then there is nothing to do, otherwise, without loss of generality we can assume that $F\cap G$ is not a join-irreducible filter, thus, by Proposition \ref{5.6}, we
have $F\cap G=\bigvee_{i=1}^{n}F_{i}$,
where $F_{i}\in\mathcal{JF}(L)$. It follows that $F_{i}\subset F\cap G$, whence $(\mathcal{JF}(L),\supseteq)$ is a directed poset. Let $F$ be a filter of $L$. Then, by Proposition \ref{5.6}, there are join-irreducible filters
$J_{1},\ldots,J_{n}$ such that $F=\bigvee_{i=1}^{n}J_{i}$. Thus there some join-irreducible filter $J_{j}$ such that $J_{j}\subseteq F$. Therefore $(\mathcal{JF}(L),\supseteq)$ is cofinal in $(\mathcal{F}(L),\supseteq)$.
\end{proof}


\begin{theorem}\label{5.11}
Let $L$ be a finite residuated lattice. Then $$\widehat{L}=\varprojlim_{F\in \mathcal{JF}(L)\cup\{1\}}L/F.$$
\end{theorem}

\begin{proof}
It follows from Proposition \ref{3.3} and Lemma \ref{5.10}.
\end{proof}

\section{The existence of Hausdorff topological residuated lattices}
In this section, we investigate linear topological residuated lattices and discuss the existence of Hausdorff topological residuated lattices.\\

For a filter $F$ of a residuated lattice $L$. We can get a zero-dimensional linear topological residuated lattice $(L,\mathcal{T}_{F})$, where $\mathcal{T}_{F}=\{U\subseteq L: \forall x\in U, x/F\subseteq U\}$, see \cite{Yang}.
For simplicity, such a zero-dimensional linear topological residuated lattice is called a \emph{simple linear topological residuated lattice} when is no danger causing confusion.\\

Let $(L,\mathcal{U})$ be a zero-dimensional linear topological residuated lattice. Considering $X_{\mathcal{U}}=\{F: 1\in F\in \mathcal{U}\}$, by Lemma 3.4 and Theorem 3.7 in \cite{Yang}, we have the following result.
\begin{theorem}\label{7.12}
Let $L$ be a residuated lattice. Then each zero-dimensional linear topology $\mathcal{U}$ of $L$ is the supremum of simple linear topological residuated lattices, i.e.,$$\mathcal{U}=\sup \{\mathcal{T}_{F}:F\in X_{\mathcal{U}} \}.$$
\end{theorem}

Now, considering a special system of filters of a residuated lattice $L$, that is, the set of all finite index filters, denoted it by $\mathcal{F}(L)_{fin}$. It is easy to see that $(\mathcal{F}(L)_{fin},\subseteq)$ is a
down-directed set. Thus, $(L,\mathcal{T}_{\mathcal{F}(L)_{fin}})$ is a zero-dimensional linear topological residuated lattice. Like the linear topology of abelian group\cite{Fuchs}, this linear topology in a residuated
lattice $L$ is called a \emph{finite index topology} of $L$. Recall that an algebra $A$ is \emph{finitely approximable} if $A$ is isomorphic to a subalgebra of a product of finite algebras \cite{Malcev}, the class of
finitely approximable residuated lattices is denoted by $\mathbf{Ap}\mathcal{RL}$. It follows that
$A$ is finitely approximable iff $A$ is a subdirect product of its finite homomorphic images.
Thus we have that a residuated lattice $L$ is finitely approximable if and only if $\bigcap\mathcal{F}(L)_{fin}=\{1\}$.\\

 As a consequence, we have the following characterization of finitely approximable residuated lattices by
using zero-dimensional linear topological residuated lattices.

\begin{theorem}\label{8.2}
A residuated lattice is finitely approximable if and only if its finite index topology is Hausdorff.
\end{theorem}

\begin{proof}
It follows from the above statement and Theorem \ref{4.7}.
\end{proof}

Let $(L,\mathcal{U})$ be a linear topological residuated lattice. Then there exists a base $\beta$ for the
topology $\mathcal{U}$ on $L$ such that any element containing 1 of $\beta$ is a filter of $L$. Let $X$ be the set of all elements containing
1 of $\beta$, $Y$ be the set of all finite intersections of elements of $X$. Then $Y$ is a system of filters of $L$. The topology induced by the system $Y$ of filters is denoted by $\mathcal{V}$.\\

In what follows, we give a simple proof of Theorem 3.7 in \cite{Yang}, which answered an open problem in \cite{Zahiri}.
\begin{theorem}\label{7.13} If $(L,\mathcal{U})$ be a zero-dimensional linear topological residuated lattice, then $\mathcal{V}$ is finer than $\mathcal{U}$.
\end{theorem}

\begin{proof} Let $\beta'=\{x/F_{i}:i\in I,x\in L\}$ be a base for topologies $\mathcal{V}$ on $L$. Since $(L,\mathcal{U})$ is a zero-dimensional space, there exists a clopen base $\gamma$ for $\mathcal{U}$. To complete the proof, we need only show that for any $x\in U\in \gamma$ there exists $x/F_{i}\in \beta'$ such that $x/F_{i}\subseteq U$. Let $x\in U\in \gamma$.
Since $U$ is clopen and $\mathcal{U}$ is finer than $\mathcal{V}$, one can obtain that $U$ is also clopen for $\mathcal{V}$. Thus there is a clopen $V$ of $\mathcal{V}$ such that $x\in V\subseteq U$. Therefore, there is some $F_{i}$ such that $x/F_{i}\subseteq V\subseteq U$, since $\beta'$ is a base for $\mathcal{V}$.
\end{proof}

\begin{corollary}\label{7.14}
The topology of a zero-dimensional linear topological residuated lattice is completely determined by a fundament system of neighborhoods of 1.
\end{corollary}

\begin{proof}
Let $(L,\mathcal{U})$ be a zero-dimensional linear topological residuated lattice. By Lemma 3.1 in \cite{Yang3}, $X_{\mathcal{U}}=\{F: 1\in F\in \mathcal{U}\}$ is a fundamental system of neighborhoods of 1. By
Corollory 3.8 in \cite{Yang1}, it follows that the topology $\mathcal{U}$ is induced by the system $X_{\mathcal{U}}$ of filters of $L$.
\end{proof}

\begin{corollary}\label{7.15}
A zero-dimensional linear topological residuated lattice is precisely the linear topological residuated lattice induced by its the system of filters.
\end{corollary}

\begin{proof}
According to Theorem 3.14 in \cite{Yang}, every topological residuated lattice induced by a system of filters is a zero-dimensional linear topological residuated lattice. The rest of proof follows from Theorem \ref{7.13} in the paper and
Corollory 3.8 in \cite{Yang1}.
\end{proof}

\begin{definition}\label{7.17}
\emph{Let $L$ be a residuated lattice and $(\mathcal{F}_{1},\subseteq)$ and ($\mathcal{F}_{2},\subseteq)$ be two systems of filters of $L$. We call that $(\mathcal{F}_{1},\subseteq)$ is \emph{equivalent} to
$(\mathcal{F}_{2},\subseteq)$, if posets $(\mathcal{F}_{1},\subseteq)$ and $(\mathcal{F}_{2},\subseteq)$ are coinitial each other, that is, for any $F_{1}\in \mathcal{F}_{1}$ there exists $F_{2}\in\mathcal{F}_{2}$ such that $F_{1}\subseteq F_{2}$ and
 vice versa.}
\end{definition}

\begin{proposition}\label{7.18}
In a residuated lattice $L$, equivalent systems of filters have the same topology.
\end{proposition}

\begin{proof}
The proof is straightforward.
\end{proof}

\begin{example} \label{4.5} Let $\odot$ and $\rightarrow$ on the real unit interval $I=[0,1]$ be defined as follows:
\begin{center} $x\odot y=min\{x,y\}$ and
$x\rightarrow y=
\begin{cases}
1, &  x\leq y,\\
y, & otherwise.\\
\end{cases}$
\end{center}
Then $\mathcal{I}=(I,min,max,\odot,\rightarrow,0,1)$ is a $BL$-algebra (called the G\"{o}del structure). It is easily checked that $\mathcal{F}=\{(a,1]\mid a\in[0,1)\}$ is a system of filters. Since $\mathcal{F}$ satisfies the finite intersection, we get that $\mathcal{F}=\mathcal{G}$, which is the set of all finite intersections of elements of $\mathcal{F}$. One can check that $L$ does not satisfy $DCC$.
\end{example}

Using the above example, we can show that the proof of Theorem 3.10 in  \cite{Yang} is incorrect. In fact, in the proof of Theorem 3.10 in  \cite{Yang}, suppose that $(L,\mathcal{T})$ is a linear topological residuated lattice such that $(L,\mathcal{T})$ is a zero-dimensional space. Then there exists a base $\beta$ for the topology $\mathcal{T}$ on $L$ such that any element containing 1 of $\beta$ is a filter of $L$. Let $X$ be the set of all elements containing 1 of $\beta$ and $Y$ be the set of all intersections of elements of $X$. Then $(X,\subseteq)$ and $(Y,\subseteq)$ are two systems of filters of a residuated lattice $L$. In general, they are not equivalent. Then the topologies induced by them are not same. For example,
taking $X=\mathcal{F}$ in Example \ref{4.5} and using the construction of $Y$, we have $\{1\}\in Y$. Thus, $\mathcal{T}_{Y}$ is a discrete topology. But $\mathcal{T}_{X}$ is not a discrete topology. Thus, $\mathcal{T}_{X}\neq\mathcal{T}_{Y}$, that is, $\mathcal{T}\neq\mathcal{V}$ in the proof of Theorem 3.10 in  \cite{Yang}. It follows that the mapping $\varphi$ from $\mathcal{F}(L)$ to $ZLTRL(L)$ defined by $\varphi(F)=(L,\mathcal{T}_{F})$ for all $F$ in $\mathcal{F}(L)$ is not surjective, where the set
$ZLTRL(L)$ of all zero-dimensional linear topological residuated lattices on a residuated lattice $L$, see \cite{Yang}.\\

Next, we will correct Theorem 3.10 in  \cite{Yang} in the following theorem.

\begin{theorem}\label{7.16}
Let $L$ be a residuated lattice. Then $|\mathcal{F}(L)|=|ZLTRL(L)|$ if and only if $L$ satisfies the descending chain condition.
\end{theorem}

\begin{proof}
Suppose that $L$ satisfies the descending chain condition. Let us define a mapping $\varphi$ from $\mathcal{F}(L)$ to $ZLTRL(L)$ by $\varphi(F)=(L,\mathcal{T}_{F})$ for all $F$ in $\mathcal{F}(L)$, where $\mathcal{T}_{F}=\{U\subseteq L: \forall x\in U, x/F\subseteq U\}$. It follows that $(L,\mathcal{T}_{F})$ is a zero-dimensional linear topological residuated lattice. According to Remark 3.9 in \cite{Yang}, the mapping $\varphi$ is injective. Thus, it suffices to show that $\varphi$ is a surjective mapping. Suppose that $(L,\mathcal{U})$ is a zero-dimensional linear topological residuated lattice. By Theorem \ref{7.13} and its corollary, there is a system of filters $X_{\mathcal{U}}$ such that $\mathcal{U}=\sup \{\mathcal{T}_{F}:F\in X_{\mathcal{U}} \}$. By hypothesis and Proposition \ref{4.4}, there is a minimum filter $F$ in $X_{\mathcal{U}}$. Therefore, we have $\mathcal{U}=\mathcal{T}_{F}$.

Conversely, assume that $|\mathcal{F}(L)|=|ZLTRL(L)|$. By Proposition \ref{4.4}, it suffices to show that each system of filters of $L$ has a minimum element. Suppose that $(\mathcal{F},\subseteq)$ is a system of filters of $L$. Then the system $(\mathcal{F},\subseteq)$ can induce a zero-dimensional linear topological residuated lattice $(L,\mathcal{T}_{\mathcal{F}})$, and $\mathcal{T}_{\mathcal{F}}=\sup\{\mathcal{T}_{F}:F\in\mathcal{F}\}$. By hypothesis, it follows that the size of zero-dimensional linear topological residuated lattices of $L$ is equal to the size of simple linear topological residuated lattices of $L$. Thus there is a filter $G\in \mathcal{F}(L)$ such that $\mathcal{T}_{G}=\sup\{\mathcal{T}_{F}:F\in\mathcal{F}\}$. It follows that $\mathcal{T}_{F}\subseteq \mathcal{T}_{G}$, thus we obtain $G\subseteq F$ for any $F\in \mathcal{F}$. Particularly, $\mathcal{T}_{G}$ is an Alexandrov topology and $1/G$ is a minimum open set of 1. According to Lemma 3.3 in \cite{Yang}, it follows that there are filters $F_{1},F_{2},\ldots,F_{n}$ in $\mathcal{F}$ such that $1/G=1/\cap_{i=1}^{n}F_{i}$. Therefore, $G=\cap_{i=1}^{n}F_{i}$. Since $(\mathcal{F},\subseteq)$ is a down-directed set, there is a filter $F\in \mathcal{F}$ such that $F\subseteq \cap_{i=1}^{n}F_{i}=G$. Therefore, we have $G=F\in \mathcal{F}$ and $G$ is a minimum filter of $\mathcal{F}$.
\end{proof}

\begin{lemma}\label{7.19}
Let $(L,\mathcal{U})$ be a topological residuated lattice. If $F$ is a filter of $L$, then $F$ is open(closed) if and only if $a/F$ is open(closed) for any $a\in L$.
\end{lemma}

\begin{proof}
There is no loss of generality in assuming that $F$ is an open filter of $L$. Let $a\in L$. We define two maps $l_{a}:L\rightarrow L$ and $g_{a}:L\rightarrow L$ by $l_{a}(x)=a\rightarrow x$ and $g_{a}(x)=x\rightarrow a$
for any $x\in L$. Then clearly, $l_{a}$ and $g_{a}$ are continuous maps and so $l_{a}^{-1}(F)$, $g_{a}^{-1}(F)$ are open in $(L,\mathcal{U})$. Since $a/F=\{x\in L:x\rightarrow a\in F ~and~a\rightarrow x\in F\}=l_{a}^{-1}(F)\cap g_{a}^{-1}(F)$, then we have $a/F\in\mathcal{U}$. Conversely, it follows the fact $F=1/F$.

\end{proof}

We state next an easy consequence of compactness which is similar to topological groups \cite{Ribes}.

\begin{theorem}\label{7.20}
In a compact zero-dimensional linear topological residuated lattice $(L,\mathcal{U})$, a filter $G$ is open if and only if $G$ is closed of finite index.
\end{theorem}

\begin{proof}
By hypothesis, let $X_{\mathcal{U}}=\{F: 1\in F\in \mathcal{U}\}$. Suppose that $G$ is an open filter of $L$, then there exists a filter $F$ in $X_{\mathcal{U}}$ such that $F\subseteq G$. Let $L=\bigcup_{x\in L}x/F$.
Since $(L,\mathcal{U})$ is compact, there are $x_{1},\ldots,x_{n}\in L$ such that $L=\bigcup_{i=1}^{n}x_{i}/F$. Due to $F\subseteq G$, it follows that $x_{i}/F\subseteq x_{i}/G$. Thus we have $L=\bigcup_{i=1}^{n}x_{i}/G$.
Without loss of generality we can assume $x_{1}/G=1/G=G$, hence we have $G=L\setminus\bigcup_{i=2}^{n}x_{i}/G$. By Lemma \ref{7.19}, each $x_{i}/G$ is open, so $G$ is closed of finite index. Conversely, suppose $G$ is a
closed filter, then we can let $L=G\cup\bigcup_{i=1}^{n-1}x_{1}/G$, by Lemma \ref{7.19}, $G=L\setminus\bigcup_{i=1}^{n-1}x_{1}/G$ is open. This proves the theorem.
\end{proof}

\begin{proposition}\label{7.21}
Let $\mathcal{F}(L)\setminus\{1\}$ be a global system of filters of a residuated lattice $L$. If $(L,\mathcal{T}_{\mathcal{F}(L)\setminus\{1\}})$ is compact and Hausdorff, then $L$ is finitely approximable.
\end{proposition}

\begin{proof}
It follows from Theorem \ref{4.7} and Theorem \ref{7.20}.
\end{proof}

\begin{theorem}\label{7.22}
Let $L$ be a non-trivial residuated lattice satisfying the descending chain condition. Then $L$ is subdirectly irreducible if and only if it has a largest non-trivial zero-dimensional linear topological residuated lattice.
\end{theorem}

\begin{proof}
It follows from Theorem \ref{7.16} in the paper and Theorem 8.4 in \cite{Burris}.
\end{proof}

By a chain of prime filters of a residuated lattice $L$ we mean a finite strictly increasing sequence $P_{0}\subset P_{1}\subset\cdots\subset P_{n}$, the \emph{length} of the chain is $n$.
\begin{definition}\label{7.23}
\emph{We define the \emph{dimension} of a residuated lattice $L$ to be the supremum of the lengths of all chains of prime filters in $L$: it is an inter $n\geq 0$, or $+\infty$(assuming $L$ is nontrivial). }
\end{definition}

Noth that such a dimension also was called the \emph{Krull dimension} in commutative rings (see for instance \cite{Atiyah}).
\begin{proposition}\label{7.24}
If a residuated lattice $L$ is subdirectly irreducible, then it does not have a non-trivial Hausdorff zero-dimensional linear topological residuated lattice.
\end{proposition}

\begin{proof}
Let $L$ be a subdirectly irreducible residuated lattice. Suppose that $(L,\mathcal{U})$ is a Hausdorff zero-dimensional linear topological residuated lattice. By Corollary \ref{7.15}, there is a system of filters $(\mathcal{F},\subseteq)$ such that $\mathcal{U}=\sup \{\mathcal{T}_{F}:F\in X_{\mathcal{U}} \}$. According to Theorem \ref{4.7} and $\mathcal{U}$ is Hausdorff, we have $\bigcap\mathcal{F}=\{1\}$. Since $\{1\}$ is completely meet-irreducible in $(\mathcal{F}(L),\subseteq)$, it follows that $\{1\}\in \mathcal{F}$. Therefore, $\mathcal{U}$ is a discrete topology.
\end{proof}

\begin{theorem}\label{7.25}
If a residuated lattice $L$ is not subdirectly irreducible and its dimension is infinite, then there must be a non-trivial Hausdorff topological residuated lattice on $L$.
\end{theorem}

\begin{proof}

Assume that $L$ is not subdirectly irreducible and its dimension is infinite. According to Corollary \ref{4.10}, we have $\bigcap\mathcal{P}(L)=\{1\}$. Since $L$ is not subdirectly irreducible, we must have $$\bigcap(\mathcal{P}(L)\setminus\{1\})=\bigcap(\mathcal{F}(L)\setminus\{1\})=\{1\},$$ otherwise it follows that $\{1\}\subset F=\bigcap(\mathcal{F}(L)\setminus\{1\})$, by Proposition \ref{4.9}, $F$ is a minimum filter in $\mathcal{F}(L)\setminus\{1\}$, which is a contradiction. Since $L$ is an infinite dimension, there is a strictly increasing sequence of $\mathcal{P}(L)\setminus\{1\}$, denote it by $\mathcal{P}$, such that $$\bigcap\mathcal{P}=\bigcap(\mathcal{P}(L)\setminus\{1\})=\{1\}.$$ It is easily checked that $(\mathcal{P},\subseteq)$ is a system of filters, thus $(L,\mathcal{T}_{\mathcal{P}})$ is a topological residuated lattice. According to Proposition \ref{4.7}, $\mathcal{T}_{\mathcal{P}}$ is Hausdorff, which is not-trivial, since the intersection of any finite of elements of $\mathcal{P}$ is not equal to $\{1\}$.

\end{proof}

\begin{example}\label{7.26}
In Example \ref{4.5}, $(\mathcal{F},\subseteq)$ is a system of filters. Since the G\"{o}del structure $\mathcal{I}=(I,min,max,\odot,\rightarrow,0,1)$ is a linear residuated lattice, it follows that each element of
$\mathcal{F}$ is a prime filter, and then $\mathcal{I}$ is an infinite dimension residuated lattice. It is well known that $\mathcal{I}$ is not subdirectly irreducible(e.g.,see \cite{Esakia} Theorem 2.4.11). Thus, by Theorem
\ref{7.25}, we have $\mathcal{T}_{\mathcal{F}}$ is a non-trivial Hausdorff zero-dimensional linear topological residuated lattice. Routine calculation shows that for any element $(a,1]$ in $\mathcal{F}$, we have
$x/(a,1]=\{x\}$ if $x\leq a$, and $x/(a,1]=(a,x]$ if $x>a$, that is, the topology is $$\mathcal{T}_{(a,1]}=\{\{x\}:x\leq a\}\cup (a,1].$$ Since $\mathcal{T}_{\mathcal{F}}=\sup\{\mathcal{T}_{(a,1]}:a\in [0,1)\}$, it follows
that the Sorgenfry topology on $I$ is coarse than the topology $\mathcal{T}_{\mathcal{F}}$.
\end{example}


\section{Conclusions and future research topics}
In this paper, we concluded that the equipotent of a profinite residuated lattice is either finite or uncountable which used to judge whether the topological residuated lattice is profinite with set-theoretical methods. Giving a profinite algebra $A$, whether there is an algebra $B$ such that
$A\cong \widehat{B}$. In lattice theory, this problem is related to Gr\"{a}tzer's clebrated problem of representable posets, see \cite{Gratzer} Problems 34-35, P 156. In this paper, we concluded that
$\mathbf{Pro}\mathcal{RL}\subseteq \mathbf{Ap}\mathcal{RL}\cap \mathbf{C}\mathcal{RL}$, where $\mathbf{C}\mathcal{RL}$ is the class of complete residuated lattices. Likewise, the results for Heyting
algebras \cite{Bezhanishvili2},
Gr\"{a}tzer's problem of representability for residuated lattices would be our further research topics.
It is well known that
$\mathbf{Pro}\mathcal{RL}\subseteq \mathbf{Sto}\mathcal{RL}$, consisting of these Stone topological residuated lattices, in other words,
the category $\mathbf{Pro}\mathcal{RL}$ is equivalent to a full subcategory of $\mathbf{Sto}\mathcal{RL}$. The question arises by P.T. Jonstone (see \cite{Johnstone} Sec. VI 2.6): given an algebra $A$, when can we say that
every Stone topological algebra is profinite? This question is also true for a wide class of algebras, including groups, rings, bounded distributive lattices, Heyting algebras and Boolean algebras.
We would study this question for residuated lattices.\\

\medskip
\noindent\textbf{Acknowledgments}
\medskip\\
\indent This research is supported by a grant of National Natural Science Foundation of China (12171294,11901371,12001423), Postdoctoral Science Foundation of China (2019M660054XB).



\begin{thebibliography}{00}































\bibitem{Atiyah} M.F. Atiyah, I.G. MacDonal, Introduction to Commutative Algebra, Addison Wesley, Reading, Mass, 1699.

\bibitem{Banaschewski} B. Banaschewski, On profinite universal algebras. In General topology and its relations to modern analysis and algebra, III(Proc. Third Prague Topological Sympos.,1971), Academia, Prague, 1972, pages
51-62.

\bibitem{Belohlavek} R. Belohlavek, Some properties of residuated lattices, Czechoslovak Mathematical Journal, 2003, 53(123):161-171.

\bibitem{Bezhanishvili}  G. Bezhanishvili, N. Bezhanishvili, Profinite Heyting algebrtas, Order, 2008, 25:211-227.

\bibitem{Bezhanishvili2} G. Bezhanishvili, N. Bezhanishvili, T. Moraschini and M. Stronkowski. Profiniteness and representability of spectra of Heyting algebras, Advances in Mathematics, 2021, 391.

\bibitem{Birkhoof} G. Birkhoof, Moore-Smith convergence in general topology, Ann. Math., 1937, 38: 39-56.

\bibitem{Blyth} T.S. Blyth, Lattice and Ordered Algebraic Structures, Springer-Verlag, London, 2005.

\bibitem{Blok} W. Blok, D. Pigozzi, Algebraizable Logics, Mem. Amer. Math. Soc. 396, Providence, 1989.

\bibitem{Burris} S. Burris, H. P. Sankappanavar, A Course in Universal Algebra, Springer New York, Graduate Texts in Mathematics, 1981.

\bibitem{Borzooei} R.A. Borzooei, G.R. Rezaei, N. Kouhestani, Metrizability on (semi)topological BL-algebras, Soft Comput. 2012, 16:1681-1690.

\bibitem{Borzooei1} R.A. Borzooei, G.R. Rezaei, N. Kouhestani, Separation axioms in (semi)topological quotient BL-algebras, Soft Comput. 2012, 16:
1219-1227.

\bibitem{Chagrov} A. Chagrov, M. Zakharyaschev, Modal Logic, Oxford Logic Guides 35, Oxford Univ. Press, Oxford, 1997.

\bibitem{Choe} T.H. Choe, Greechie R J, Profinite orthomular lattices, Proceedings of AMS, 1993, 118: 1053-1060.


\bibitem{Dimitric} R. Dimitric, Slenderness, Volume 1: Abelian Categories,  Cambridge University Press, Cambridge, 2019.

\bibitem{Esakia} L. Esakia, Heyting Algebras. Duality Theory, Springer, 2019, English translation of the original 1985 book.

\bibitem{Fuchs} L. Fuchs, Abelian Groups. Springer Monographs in Mathematics. Springer, Cham, 2015.

\bibitem{Galatos} N. Galatos, P. Jipsen, T. Kowalski, H. Ono, Residuated Lattices: An Algebraic Glimpse at Substructural Logic, Studies in Logic and the Foundations of Mathematics, Elseier, 2007.

\bibitem{Gratzer} G. Gr\"{a}tzer, Lattice Theory. First Concepts and Distributive Lattices, W.H. Freeman and Company, San Francisco, 1971.

\bibitem{Gratzer1} G. Gr\"{a}tzer, Universal Algebra, 2end ed.,  Springer, 2008.

\bibitem{Yang} P.F. He, J. Yang, J.T. Wang, Solutions to open problems in topological residuated lattices, Fuzzy Sets and Systems, 2021, 405: 65-73.

\bibitem{Hohle}  U. H\"{o}hle, Commutative residuated $\ell$-monoids. In: U. H\"{o}hle, and E. Klement, (eds), Non-Classical Logics and their Applications to Fuzzy Subsets. Kluwer, Dordrecht, 1995.

\bibitem{Hull} M. Hull, A topology for free groups and related groups, Ann. Math., 1950, 52: 127-139.

\bibitem{Jam} I.M. Jam, Introduction to Uniform Spaces, London Mathematical Socirty Lecture Notes Series, Volume 144, Cambridge University press, Cambridge, 1999.

\bibitem{Johnstone} P.T. Johnstone, Stone Spaces, Cambridge University Press, Cambridge, 1982.

\bibitem{Kondo} M. Kondo, E. Turunen, Prime filters on residuated lattices, In IEEE 43rd International Symposium on Multiple-Valued Logic, 2013, 89-91.


\bibitem{Luan} W. Luan, H. Weber, Y.C. Yang, Filter topologies and topological MV-algebras, Fuzzy Sets and Systems, 2021, 406: 11-21.

\bibitem{Luan1} C. Luan, Y.C. Yang, Filter topologies on MV-algebras, Soft Comput. 2017, 21: 2531-2535.


\bibitem{Malcev} A.I. Malcev, Algebraic Systems, Springer, New York, 1973.



\bibitem{Najafi} M. Najafi, G.R. Rezaei, N. Kouhestani, On (para, quasi) topological MV-algebras, Fuzzy Sets and Systems, 2017, 313: 93-104.


\bibitem{Nganou} J.B. Nganou, Profinite $MV$-algebras and Multisets, Order, 2015, 32: 449-459.

\bibitem{Nikolov} N. Nikolov, D. Segal, On finitely generated profinite groups, I: strong completeness and uniform bounds, Ann. Math., 2007, 165: 171-238.

\bibitem{Nikolov1} N. Nikolov, D. Segal, On finitely generated profinite groups, II: products in quasisimple groups, Ann. Math., 2007, 165: 239-273.

\bibitem{Ribes} L. Ribes, P. Zalesskii, Profinite Groups. Second edition, Results in Mathematics Related Areas, 3rd Series, A Series of Modern Surveys in Mathematics 40, Springer-Verlag, Berlin, 2010

\bibitem{Schneider} F.M. Schneider, J. Zumbr\"{a}gel, Profinite algebras and affine boundedness, Advances in Mathematics, 2017, 305: 661-681.

\bibitem{Turunen} E. Turunen, Mathematics Behind Fuzzy Logic, Physica-Verlag, 1999.

\bibitem{Ward1} M. Ward, P.R. Dilworth, Residuated lattice, Trans. Am. Math. Soc., 1939, 45: 335-354




\bibitem{Yan} Y.C. Yang, The C-topology on lattice-ordered groups, Sci. China Ser. A, Math. 2009, 52: 2397-2403.



\bibitem{Yang2} J. Yang, X.L. Xin, P.F. He, On topological EQ-algebras, Iranian Journal of Fuzzy Systems, 2018, 15(6): 145-158.



\bibitem{Yang3} J. Yang, X.L. Xin, P.F. He, Notes on topological $BL$-algebras, Fuzzy Sets and Systems, 2018, 350: 33-40.

\bibitem{Yang1} J. Yang, X.W. Zhang, Some weaker versions of topological residuated lattices, Fuzzy Sets and Systems, 2019, 373: 62-77.

\bibitem{Zahiri} O. Zahiri, R.A. Borzooei, Topology on $BL$-algebras, Fuzzy Sets and Systems, 2016, 289: 137-150.




























\end{thebibliography}
\end{document}